\documentclass[a4paper,12pt]{article}
\usepackage[utf8]{inputenc}
\usepackage[english,french]{babel}
   \usepackage{hyperref,cleveref,varioref}
   \usepackage{amsmath}
   \usepackage{amsfonts}
   \usepackage{amssymb}
   \usepackage{mathrsfs,mathtools}
   \usepackage{graphicx}
   \usepackage{textcomp}
   \usepackage{etex}
   \usepackage[abs]{overpic}
   \usepackage{float}
   \usepackage{cases}
   \usepackage{prettyref}
   \usepackage{empheq}
\newcommand{\R}{\mathbb{R}}
\newcommand{\C}{\mathbb{C}}
\newcommand{\N}{\mathbb{N}}

\newtheorem{thm}{Theora}[section]
\newtheorem{Theo}[thm]{Theorem}

\newtheorem{Cor}[thm]{Corollary}
\newtheorem{Lem}[thm]{Lemma}
\newtheorem{Prop}[thm]{Proposition}

\newtheorem{Rem}[thm]{Remark}

\newenvironment{theorem*}[1]{\smallskip\noindent{\bf #1.}\rm}{\medskip}

\newenvironment{Proof}{\smallskip\noindent{\bf Proof.}\rm}
{\hfill $\Box$\medskip}
\newenvironment{proof}{\smallskip\noindent{\bf Proof}\rm}
{\hfill $\Box$\medskip}
\numberwithin{equation}{section}
\renewcommand\({\left(}
\renewcommand\){\right)}
\renewcommand\[{\left[}
\renewcommand\]{\right]}

\newcommand\Dim{{\rm{dim}}}

\newcommand\la{\lambda}
\newcommand\ga{\gamma}
\newcommand\rh{\rho}
\newcommand\si{\sigma}

\newcommand{\be}{\begin{equation}}
\newcommand{\ee}{\end{equation}}
\newcommand{\ba}{\begin{array}}
\newcommand{\ea}{\end{array}}
\newcommand{\bea}{\begin{eqnarray*}}
\newcommand{\eea}{\end{eqnarray*}}
\newcommand{\bean}{\begin{eqnarray}}
\newcommand{\eean}{\end{eqnarray}}

\newcommand\se{\sigma}

\makeatletter \@addtoreset{equation}{section}

\makeatother
\begin{document}
\title{Exact Boundary Controllability for the
Boussinesq Equation with Variable Coefficients}

\author{Jamel Ben Amara \thanks{Department of Mathematics, Faculty of Sciences of Tunis,
 University of Tunis el Manar, Mathematical Engineering Laboratory, Tunisia; e-mail: {\url{
jamel.benamara@fsb.rnu.tn}}.}~~~~~~~~~~Hedi Bouzidi
\thanks{Department of Mathematics, Faculty of Sciences of Tunis,
 University of Tunis el Manar, Mathematical Engineering Laboratory, Tunisia; e-mail:
{\url{bouzidihedi@yahoo.fr}}.}}
\date{}
 \maketitle {\bf Abstract:} In this paper we study the
exact boundary controllability for the following Boussinesq equation
with variable physical parameters:  \bea\left\{
 \begin{array}{lll}
\rho(x)y_{tt}=-(\sigma(x)y_{xx})_{xx}+(q(x)y_x)_x-(y^2)_{xx},&&t>0,~x\in(0,l),\\
y(t,0)=\si(l)y_{xx}(t,0)=y(t,l)=0,~~\si(l)y_{xx}(t,l)=u(t)&&t>0,\\
 \end{array}
 \right.
 \eea
where $l>0$, the coefficients $\rho(x)>0,\si(x)>0 $, $q(x)\geq0$ in
$\[0,l\]$ and $u$ is the control acting at the end $x=l$. We prove
that the linearized problem is exactly controllable in any time
$T>0$. Our approach is essentially based on a detailed spectral
analysis together with the moment method. Furthermore, we establish
the local exact controllability for the nonlinear problem
by fixed point argument.\\



\section{Introduction}
\indent  Let $T>0$ and $l >0$. The classical Boussinesq equation on
the bounded domain $\(0,l\)$ is of the form \be\label{crepeau1}
y_{tt}+ \sigma y_{xxxx}- y_{xx}+\(y^{2}\)_{xx}=0,~~~~(t,x)\in (0, T
)\times(0, l), \ee where the coefficient $\sigma\in\R$. This
equation was derived by the French mathematician Joseph Boussinesq
\cite{B} in 1872 as a model for the propagation of small amplitude
of long waves on the surface of water. This was the first to give a
scientific explanation of the existence of solitary waves found by
Scott Russell's in the 1840's. Depending on whether the coefficient
$\sigma$ in \eqref{crepeau1} is positive or negative, Equation
\eqref{crepeau1} nowadays known as the "good" or the "bad"
Boussinesq equation. In fact, the "bad" Boussinesq equation is
linearly unstable and admits the inverse scattering approach
\cite{DTT,Z}. For this reason, we only consider the version of the
"good" Boussinesq equation with variable coefficients. These
equations arise as a model of nonlinear vibrations along a beam
\cite{Z},
and also for describing electromagnetic waves in nonlinear dielectric materials \cite{T}.\\
\indent From a mathematical point of view, well-posedness and
dynamic properties of "good" Boussinesq equations have a huge
literature, see the paper by Bona and Sachs \cite{BS}, see also
\cite{A, L} and references therein.\\
\indent We are concerned with the exact boundary controllability for
the "good" Boussinesq equation with variable coefficients. In this
direction, the case of the linear "good" Boussinesq equation with
constant coefficients has been investigated by Lions in \cite{J.L2}.
In that reference, by Hilbert Uniqueness Method "Lions'{\rm HUM}"
(cf. Lions \cite{J.L1, J.L2}), it was proved that the linearized
"good" Boussinesq system
\begin{empheq}[left=\empheqlbrace]{align}
&y_{tt}=- y_{xxxx}+y_{xx},~~~~~~~~~~~~~~~~~(t,x)\in(0, T )\times(0, l),\label{crepeau2}\\
&y(t, 0)=y_{xx}(t, 0)=0,~~~~~~~~~~~~~~t\in(0,T),\label{crepeau3}\\
&y(t, l) =\tilde{u}(t),~y_{xx}(t, l) =u(t),~~~~t\in(0,T),\label{crepeau4}\\
&y(0, x) = y^{0},~~y_{t}(0, x) =
y^{1},~~~~~~~x\in(0,l),\label{crepeau5}
\end{empheq}
 is exactly controllable in any time
$T>2(l+\frac{1}{\sqrt{\la_0}})$, where the two controls $\(\tilde
u,~u\)\in L^2(0,T)\times H^{1}(0,T)$ and $\la_0$ denotes the first
eigenvalue of the operator $-\Delta$ with the Dirichlet boundary
conditions. Later on, Zhang \cite{ZH} studied the question of
distributed control for the generalized "good" Boussinesq equation
with constant coefficients on a periodic domain. A few years after,
Cr\'{e}peau extended in \cite{C} the results obtained in
\cite{J.L2}. More precisely, by a detailed spectral analysis and the
use of nonharmonic Fourier series, it was shown that System
\eqref{crepeau2}-\eqref{crepeau5} (for $\tilde u\equiv0$) is exactly
controllable at any time $T>0$, where the control $u\in L^2(0,T)$.
Furthermore, with the help of the fixed point theorem, it was also
proved the local controllability for the nonlinear control problem
\eqref{crepeau1}, \eqref{crepeau3} and \eqref{crepeau4} (for
$\si\equiv1$ and $\tilde u\equiv0$). Concerning the control of the
approached "bad" Boussinesq equation, the controllability properties
for the so called improved Boussinesq equation have been
obtained recently by Cerpa and Cr\'{e}peau \cite{CK*}.\\
\indent In the previous studies of controllability, the coefficients
of the "good" Boussinesq equation are supposed to be constant. In
the present paper, we address the problem of exact boundary
controllability for the "good" Boussinesq equation with variable
coefficients. More precisely, we consider the following control
problem \bean \label{crepeau6} \left\{
 \begin{array}{ll}
\rho(x)y_{tt}=-(\sigma(x)y_{xx})_{xx}+(q(x)y_x)_x-\(y^2\)_{xx},&(t,x)\in
(0, T )\times(0, l),\\
y(t, 0) =\si(l)y_{xx}(t,0) =y(t, l) = 0,~~\si(l)y_{xx}(t, l) = u(t),&t\in(0,T),\\
y(0, x) = y^{0},~~y_{t}(0, x) = y^{1},& x\in(0,l),\\
\end{array}
 \right.
\eean where $u$ is a control placed at the extremity $x = l$, and
the functions $y^0$, $y^1$ are the initial conditions. Here and in
what follows, we assume that the coefficients \be \rho,\sigma \in
H^{3}(0,l),~q\in H^{1}(0,l), \label{crepeau7}\ee and there exist
constants $\rho_0,~\sigma_0>0$, such that \be
\rho(x)\geq\rho_{0},~~\se(x)\geq\se_{0},~~ q(x)\geq0,~~x \in
\[0, l\]. \label{crepeau8}\ee
\indent In this paper we prove that the linearized problem \bean
\label{crepeau9} \left\{
 \begin{array}{ll}
\rho(x)y_{tt}=-(\sigma(x)y_{xx})_{xx}+(q(x)y_x)_x,&(t,x)\in(0,T)\times(0,l),\\
y(t, 0) =\si(l)y_{xx}(t,0) =y(t, l) = 0,~~\si(l)y_{xx}(t, l) = u(t),&t\in(0,T),\\
y(0, x) = y^{0},~~y_{t}(0, x) = y^{1},& x\in(0,l),\\
\end{array}
 \right.
\eean is exactly controllable in any time $T>0$, where the control
$u\in L^2(0,T)$ and the initial conditions $\(y^{0},y^1\)$ taken in
$H^1_0(0,l)\times H^{-1}(0,l)$. Our approach is essentially based on
a detailed spectral analysis and the qualitative theory of
fourth-order linear differential equations due
to Leighton and Nehari \cite{LN}. 
 More precisely, we prove that all the eigenfrequencies
$\(\sqrt{\lambda_{n}}\)_{n\geq1}$ associated System \eqref{crepeau9}
(without control) are simple, and by a precise computation of its
asymptotics we show that the spectral gap
"$\big{|}\sqrt{\lambda_{n+1}}-\sqrt{\lambda_{n}}\big{|}"$ is of
order $\mathcal{O}(n)$. Moreover, we prove that the first derivative
of each eigenfunction $\phi_n,~n\geq1,$ associated with uncontrolled
system does not vanish at the end $x=l$. As a consequence of the
theory of non-harmonic Fourier series and an extension of Ingham's
Theorem due to Haraux \cite{AH}, we establish the equivalence
between the $H^1_0\times H^{-1}$-norm of the initial data $\(\tilde
y^{0},\tilde y^1\)$ and the quantity $\int_{0}^{T}|\tilde
y_{x}(t,l)|^{2}dt$, where $\tilde y$ is the solution of System
\eqref{crepeau9} without control. Finally, we apply the Lions'{\rm
HUM} to deduce the
exact controllability result for the system \eqref{crepeau9}.\\
\indent At the end of this paper, we will discuss the local exact
controllability for the nonlinear control system \eqref{crepeau6}.
To this end, we use some results obtained by Cr\'{e}peau \cite{C}
together with standard fixed-point method (e.g., \cite[Chapter 4]{J}
and \cite{R}).\\
\indent The rest of the paper is divided in the following way: In
section $2$, we establish the well-posedness of System
\eqref{crepeau9} without control. In the next section, we prove the
simplicity of all the eigenvalues $\(\la_n\)_{n\geq1}$ and we
determinate the asymptotics of the associated spectral gap. In
section $4$, we prove the exact controllability result for the
linear control problem \eqref{crepeau9}. The last section is devoted
to the local controllability for the nonlinear control problem
\eqref{crepeau6}.
\section{Operator Framework and Well-posedness}
In this section we investigate the well-posedness of the linear
homogeneous Boussinesq problem, \bean \label{crepeau10} \left\{
\begin{array}{lll}
\rho(x){y}_{tt}=-(\sigma(x){y}_{xx})_{xx}+
(q(x){y}_x)_x,~&(t,x)\in(0,T)\times(0,l),\\
{y}(t, 0) =\si(l){y}_{xx}(t, 0) ={y}(t, l)
=\si(l){y}_{xx}(t, l) = 0,&t\in(0,T),\\
 {y}(0, x) = {y}^{0},~
 {y}_{t}(0, x) = {y}^{1},&x\in(0,l).
\end{array}
 \right.
\eean First of all, let us define by $L^2_\rho(0,l)$ the space of
functions $f$ such that $$\int_0^l |f|^{2}\rho(x)dx<\infty,$$ and we
denote by $H^k(0,l)$ the $L^2_\rho(0,l)-$based Sobolev spaces for
$k> 0$. We consider the following Sobolev space \be H^{2}(0,l)\cap
H^{1}_{0}(0,l) \label{crepeau11}\ee endowed with the norm $
\|u\|_{H^{2}(0,l)\cap H^{1}_{0}(0,l)}=\|u''\|_{L^2_\rho(0,l)}$. It
is easily seen from Rellich's theorem that the space $H^{2}(0,l)\cap
H^{1}_{0}(0,l)$ is densely and compactly embedded in the space
$L^2_\rho(0,l)$. In the sequel, we introduce the operator
$\mathcal{A}$ defined in $L^2_\rho(0,l)$ by setting:
 \be \mathcal{A}(y)= \rho^{-1}\((\sigma y'')''-(qy')'\), \label{crepeau12}\ee
 on the domain
 \be \mathcal D(\mathcal{A})=\{y\in H^{4}(0,l) \hbox { such that } y ,~ y''\in H^{1}_{0}(0,l) \}, \label{crepeau13}\ee
which is dense in $L^2_{\rho}(0,l)$.
\begin{Lem}\label{rr}
The linear operator $\mathcal{A}$ is positive and self-adjoint such
that $\mathcal{A}^{-1}$ is compact. Moreover, the linear operator
$\mathcal{A}^{\frac{1}{2}}$ generates a strongly continuous
semi-group on $L^{2}_{\rho}(0,l)$.
\end{Lem}
\begin{Proof}
Let $y \in\mathcal D(\mathcal{A})$, then by integration by parts, we
have \bean \langle \mathcal{A}y,
y\rangle_{L^2_\rho(0,l)}&=&\int_{0}^{l}\Big{(}(\sigma(x)y'')''-(q(x)y')'\Big{)}y dx \nonumber\\
&=&\int_{0}^{l}\sigma(x)|y''|^{2}dx+q(x)|y'|^{2}dx\label{crepeau14}
\eean
 since $\sigma>0$ and $q\geq0$ then $\langle
\mathcal{A}y, y\rangle_{L^2_\rho(0,l)}>0,$ and hence the quadratic
form has a positive real values, so the linear operator
$\mathcal{A}$ is symmetric. Furthermore, it is easy to show that
$Ran(\mathcal{A}-iId)=L^2_\rho(0,l)$, which implies that
$\mathcal{A}$ is selfadjoint. Since the space $H^{2}(0,l)\cap
H^{1}_{0}(0,l)$ is continuously and compactly embedded in the space
$L^2_\rho(0,l)$, then $\mathcal{A}^{-1}$ is compact in
$L^2_\rho(0,l)$.
\end{Proof}\\
Lemma \ref{rr} leads to the following corollary.
\begin{Cor}\label{cor}The spectrum of the operator $\mathcal{A}$ is discrete. It consists of a
 sequence of positive eigenvalues
$(\lambda_{n})_{n\in\mathbb{N}^{*}}$ tending to $+\infty$:
$$0<\lambda_{1}\leq\lambda_{2}\leq.......\leq\lambda_{n}\leq.....
\underset{n\rightarrow +\infty}{\longrightarrow}+\infty.$$ Moreover,
the corresponding eigenfunctions $(\Phi_{n})_{n\geq1}$ can be chosen
to form an orthonormal basis in $L^2_\rho(0,l)$.\end{Cor} We give
now a characterization of some fractional powers of the linear
operator $\mathcal{A}$ which will be useful to give a description of
the solutions of Problem \eqref{crepeau10} in terms of Fourier
series. According to Lemma \ref{rr}, the operator $\mathcal{A}$ is
positive and self-adjoint, and hence it generates a scale of
interpolation spaces $\mathcal{H}_{\theta}$, $\theta \in
\mathbb{R}$. For $\theta\geq0$, the space $\mathcal{H}_{\theta}$
coincides with $\mathcal D(\mathcal{A}^{\theta})$ and is equipped
with the norm $\|u\|_\theta^2=\langle \mathcal{A}^\theta u,
\mathcal{A}^\theta u\rangle_{L^2_\rho(0,l)}$, and for $\theta< 0$ it
is defined as the completion of $L^2_\rho(0,l)$ with respect to this
norm. Furthermore, we have the following spectral representation of
space $\mathcal{H}_{\theta}$,
$$\mathcal{H}_{\theta}=
\{u(x)=\sum\limits_{n\in\N\backslash\{0\}}c_n\Phi_n(x)~:
 ~\|u\|_{\theta}^2=\sum\limits_{n\in\N\backslash\{0\}}\la_n^{2\theta}|c_n|^{2}<\infty\},$$
where $\theta\in \R$, and the eigenfunctions $\(\Phi_{n}\)_{n\geq1}$
are defined in Corollary \ref{cor}. In particular,
$\mathcal{H}_{0}=L^2_\rho(0,l) \hbox{ and }
\mathcal{H}_{1/2}= H^{2}(0,l)\cap H^{1}_{0}(0,l)$.\\
Obviously, the linear problem \eqref{crepeau10} can be rewritten in
the abstract form
\begin{equation*}
\ddot{{y}}(t)+\mathcal{A}{y}(t)=0,~ (y(0),\dot
y(0))=({y}^{0},{y}^{1}),~~t\geq0,
\end{equation*}
where $\mathcal{A}$ is defined by \eqref{crepeau12}. As a
consequence of the spectral decomposition of the operator
$\mathcal{A}$ and by \cite[Theorem 1.1]{V.K}, we have the following
existence and uniqueness result for Problem \eqref{crepeau10} in the
spaces $\mathcal{H}_{\theta} \times \mathcal{H}_{\theta-1/2}$ with
$\theta\in\R$.
\begin{Prop}\label{pos}
Let $\theta\in\R$ and $({y}^0,{y}^1)\in\mathcal{H}_{\theta} \times
\mathcal{H}_{\theta-1/2}$. Then Problem \eqref{crepeau10} has a
unique solution ${y}\in C([0,T],\mathcal{H}_{\theta})\cap
C^{1}([0,T], \mathcal{H}_{\theta-1/2})$ and is given by
the following Fourier series 
\be {y}(t,x)=\sum\limits_{n\in\N\backslash\{0\}} \Big(a_n
\cos(\sqrt{\la_n}t)+
\frac{b_n}{\sqrt{\la_n}}\sin(\sqrt{\la_n}t)\Big){\Phi}_n(x),\label{sl}\ee
where ${y}^0=\sum\limits_{n\in\N\backslash\{0\}} a_n\phi_n$ and
${y}^1=\sum\limits_{n\in\N\backslash\{0\}}b_n\phi_n$. 
\end{Prop}
\section{Spectral Analysis}\label{Spe}
In this section, we shall establish the spectral proprieties related
to System \eqref{crepeau10}. To this end we need some results of the
qualitative theory of fourth-order linear differential equations due
to Leighton and Nehary \cite{LN}, see also \cite{B2, BK, BK1}.\\
We consider the following spectral problem which arises by applying
separation of variables to System \eqref{crepeau10},
\begin{empheq}[left=\empheqlbrace]{align}
&(\sigma(x)\phi'')''-(q(x)\phi')'=\lambda \rho(x) \phi,~~x\in(0,l),\label{ss}\\
&\phi(0)=\phi''(0)=\phi(l)=0,\label{ss*1}\\
&\phi''(l)=0. \label{s1}
\end{empheq}
Our first main result in this section is the following:
\begin{Theo}\label{Lem2}
All the eigenvalues $(\la_n)_{n\geq1}$ of the spectral problem
\eqref{ss}-\eqref{s1} are simple. Moreover, the corresponding
eigenfunctions $\(\Phi_n\)_{n\geq1}$ satisfy \be
\Phi_n'(l)\mathrm{T}\Phi_n(l)<0~~\hbox{for all
}n\in\N\backslash\{0\}, \label{sp30}\ee where
$\mathrm{T}\Phi_n=(\si(x)\Phi''_n)'- q(x)\Phi'_n$.
\end{Theo}
In order to prove this theorem, we need the following result.
\begin{Lem}\cite[Lemma 2.1]{LN}\label{Ls1}
Let $\phi$ be a nontrivial solution of the differential equation
\be\label{s7} (\sigma(x)\phi'')''-\rho(x) \phi=0. \ee If $\phi,
\phi', \phi''$ and $(\sigma\phi'')'$ are nonnegative at $x=a$ (but
not all zero) they are positive for all $x>a$. If $\phi, -\phi',
\phi''$ and $-(\sigma\phi'')'$ are nonnegative at $x=a$ (but not all
zero) they are positive for all $x<a$.
\end{Lem}
In the case $q\geq0$ or if the second-order equation \be\label{s4}
(\sigma(x) h')'-q(x)h=0,~~x\in(0,l), \ee has a positive solution,
they gave a transformation \cite[Theorem 12.1]{LN} for removing the
"middle term" $(q(x)\phi')'$ from Equation \eqref{ss}. Namely, 
 if $h$ is a positive solution of the equation \eqref{s4}
then the following modified substitution \cite[Theorem 12.1]{LN}
\begin{equation}\label{s6}
s(x):=\frac {l\int_0^x h(t)ds}{\int_0^l h(t)ds},
\end{equation}
transform \eqref{ss} into the equation
\begin{equation}\label{s1t}
{\(\si h^{3}(s)\ddot \varphi\)}^{..} = h^{-1}\rh(s) \varphi,~~s\in(0,l),\\
\end{equation}
where $\si (x),h(x),\rh (x)$ are taken as functions of $s$,
$^\cdot:=\frac{d}{ds}$ and $\varphi:=\phi(x(s))$. Furthermore, we
have the following relations:
\begin{equation}\label{s2t}
{\dot \varphi}=\phi'h^{-1}, \quad
 h^{3}{\ddot \varphi}=h\phi''- \phi'h', \quad
{(\si h^{3}{\ddot \varphi})}^{.}=\mathrm{T}\phi. \\
\end{equation}
We are now ready to prove Theorem \ref{Lem2}.\\
\begin{Proof} We first prove that the set $\mathcal{E}_\la$ of solutions of the following
boundary value problem
\begin {equation}\label{eq:vs}
\left\{
\begin{array}{ll}
(\se(x)\phi'')'' -(q(x)\phi')'= \la \rho(x) \phi,~~x\in(0,l), \\
\phi(0) = \phi''(0)=0,\\
\phi''(l)=0,
\end{array}
\right .
\end{equation}
is one-dimensional subspace. To do this let $h$ denotes the solution
of Equation \eqref{s4} satisfying the initial conditions
\begin{equation}\label{s5}
h(0)=1, \quad h'(0)=0.
\end{equation}
It is known, by Sturm oscillation theorem \cite[Chapter 1]{BI} that
$h(x)>0$ on $[0,l]$. Since
$$\sigma(x)h'(x)=\int_0^xq(x)h(x)\rh(x)dx,$$ we have also, $h'(x)>0$
on $]0,l]$. Furthermore, by using the transformation \eqref{s6} and
the relations \eqref{s2t}, the boundary value problem \eqref{eq:vs}
can be rewritten in the form
\begin{empheq}[left=\empheqlbrace]{align} &{\(\si h^{3}(s)\ddot
\varphi\)}^{..} =\la h^{-1}\rh(s)  \varphi,~~s\in(0,l),\label{eq:cdt1}\\
&\varphi(0) = \ddot{\varphi}(0)=0, \label{eq:cdt}\\
&h^2{\ddot \varphi(l)}=-h'\dot{\varphi}(l).
\label{eq:cct}\end{empheq} Let $\varphi_1$ and $\varphi_2$ be two
linearly independent solutions of \eqref{eq:cdt1}-\eqref{eq:cct}.
Both $ \dot \varphi_{1}(0)$ and $ \dot \varphi_{2}(0)$ are different
from zero since otherwise the first statement of Lemma~\ref{Ls1}
would imply that $\dot \varphi_{1}\ddot \varphi_{1}(l)>0$ and $\dot
\varphi_{2}\ddot \varphi_{2}(l)>0$, and this is in contradiction
with the boundary condition \eqref{eq:cct}. In view of the
assumptions about $\varphi_{1}$ and  $\varphi_{2}$, the solution
$\psi$ defined by
 $$\psi(s)= \dot \varphi_{2}(0)\varphi_{1}(s) -  \dot \varphi_{1}(0)\varphi_{2}(s)$$
  satisfies
 $\psi(0)= \dot \psi(0)= \ddot \psi(0)=0$ and $\dot \psi\ddot \psi(l)\leq0$. This again contradicts
Lemma~\ref{Ls1} unless $\psi\equiv0$, which proves that
$\Dim\,\mathcal{E}_{\la}=1$. Therefore, each eigenvalue $\la_n$
$(n\geq1)$ is geometrically simple. On the other hand, by
Proposition \ref{rr}, the operator $\mathcal{A}$ is self-adjoint in
$L^2_\rho(0,l)$, and this implies that all the eigenvalues
$\(\la_n\)_{n\geq1}$ are algebraically simple.\\
Let us now prove \eqref{sp30}. Let $\{\la_n,\Phi_{n}\}$ $(n\geq1)$
be an eigenpair of Problem \eqref{ss}-\eqref{s1}, and let
$\tilde{h}$ be the solution of Equation \eqref{s4} satisfying the
following initial conditions
\begin{equation}\label{ss5}
\tilde h(l)=1, \quad \tilde h'(l)=0.
\end{equation}
In a same way as above, by Sturm oscillation theorem \cite[Chapter
1]{BI}, one can prove that  $\tilde h>0$ and $\tilde h'>0$ in
$\[0,l\]$. Furthermore, by using the substitution \eqref{s6},
Problem \eqref{ss}-\eqref{s1} transforms into
\begin{empheq}[left=\empheqlbrace]{align}
&{\(\si \tilde h^{3}(s)\ddot \varphi\)}^{..} = \la \tilde
h^{-1}\rho(s)
 \varphi,~~s
 \in(0,l),\label{sp10}\\
&\varphi(0)=0,~~ \tilde h^{2}\ddot\varphi(0)=-\tilde h'\dot\varphi(0),\label{sp100}\\
&\varphi(l)=0,~~{\ddot \varphi(l)}=0.\label{sp1000}
\end{empheq} Suppose that for some $n\geq1$, $\Phi_n'(l)\mathrm{T}\Phi_n(l) \geq
0$. Then from the relations \eqref{s2t} we have $${\dot
\varphi}_n(l){(\si h^{3}{\ddot \varphi_n(l)})}^{.}\geq0.$$ It
follows from the second statement of Lemma~\ref{Ls1}, that
$\varphi_n(0) \neq 0$,
but this contradicts the first boundary condition in \eqref{sp100}.
Hence, $\Phi_n'(l)\mathrm{T}\Phi_n(l)< 0$ for all
$n\in\N\backslash\{0\}$. 
This finalizes the proof of the theorem.

\end{Proof}\\
Our second main result in this section establishes the asymptotic
behavior of the spectral gap $\sqrt{\la_{n+1}}-\sqrt{\la_{n}}$ for
large $n$. Namely, we enunciate the following result:
\begin{Theo}\label{SP2}  The eigenvalues $(\la_{n})_{n\geq1}$ of the
associated spectral problem \eqref{ss}-\eqref{s1} satisfy the
following asymptotics: \be {\la_n}= \(\frac{n\pi}{\int^l_0
\sqrt[4]{\frac{\rho(t)}{\sigma(t)}}
 dt}\)^{4}+\mathcal{O}\(n^2\),\label{sp8}\ee as
$n\rightarrow\infty$. Furthermore,
 \be \sqrt{\la_{n+1}}-\sqrt{\la_n}=\mathcal{O}(n).\label{g1}\ee
\end{Theo}
\begin{Proof} It is known (e.g., \cite[Chapter~5, p.235-239]{F} and
\cite[Chapter~2]{N}) that for $\la\in\mathbb{C}$, Equation
\eqref{ss} has four fundamental solutions $\{{\phi}_i(x,
\lambda)\}_{i=1}^{i=4}$
satisfying the asymptotic forms 
\bean \label{sp3}\left\{
\begin{array}{ll} {\phi}_i(x,\la)
=\(\[\rho(x)\]^{\frac{3}{4}}\[\sigma(x)\]^{\frac{1}{4}}\)^{-\frac{1}{2}}\exp\left\{\mu
w_i \int^x_0 \sqrt[4]{\frac{\rho(t)}{\sigma(t)}}
 dt\right\} [1],\\
~~\\
{\phi}_{i}^{(k)}(x,\la) = (\mu w_i)^{k}
\(\frac{\rho(x)}{\sigma(x)}\)^{\frac{k}{4}}
\(\[\rho(x)\]^{\frac{3}{4}}\[\sigma(x)\]^{\frac{1}{4}}\)^{-\frac{1}{2}}
\exp\left\{\mu w_i \int^x_0 \sqrt[4]{\frac{\rho(t)}{\sigma(t)}}
 dt\right\} [1],\end{array}
\right .\eean where ${\mu}^4=\la$, ${w_i}^4=1$,
${\phi}^{(k)}:=\frac{\partial^k{\phi}}{{\partial x^k}}$ for
$k\in\{1, 2, 3\}$, and
 $\[1\]=1 + \mathcal{O}(\mu^{-1})$ uniformly as $\mu\rightarrow\infty$ in a sector
$\mathcal{S}_\tau=\{\mu\in\C \hbox{~~such
that~~}0\leq\arg(\mu+\tau)\leq\frac{\pi}{4}\}$ where $\tau$ is any
fixed complex number. It is convenient to rewrite these asymptotes
in the form\bea
&&{\phi}_1(x,\la)=\zeta(x)\cos(\mu X)[1],\\
&&{\phi}_2(x,\la)=\zeta(x)\cosh(\mu X)[1],\\
&&{\phi}_3(x,\la)=\zeta(x)\sin(\mu X)[1],\\
&&{\phi}_4(x,\la)=\zeta(x)\sinh(\mu X)[1], \eea where \be
\zeta(x)=\(\[\rho(x)\]^{\frac{3}{4}}\[\sigma(x)\]^{\frac{1}{4}}\)^{-\frac{1}{2}}
\hbox{~~and~~} X=\int^x_0 \sqrt[4]{\frac{\rho(t)}{\sigma(t)}}
 dt.\label{c10}\ee Hence every solution $\phi(x,\la)$ of Equation
\eqref{ss} can be written in the following asymptotic form \bean
\phi(x,\la)=\zeta(x)\Big(C_1 \cos\(\mu X\) + C_2 \cosh\(\mu X\)+ C_3
\sin\(\mu X\)
+ C_4 \sinh\(\mu X\)\Big)\[1\]\label{sp15}
\eean and from \eqref{sp3}, we have also  \bean
\phi^{(k)}(x,\la)&=&\mu^k\zeta(x)
\(\frac{\rho(x)}{\sigma(x)}\)^{\frac{k}{4}}\Big{(}C_1 \cos^{(k)}(\mu
X) + C_2
\cosh^{(k)}(\mu X) + C_3 \sin^{(k)}(\mu X) \nonumber\\
 &&+ C_4 \sinh^{(k)}(\mu X)\Big{)}\[1\],~~\mbox{as}~~\mu \rightarrow\infty,~~
k\in\{1, 2, 3\}, \label{sp12} \eean
 where $C_i, i=1,2,3,4$ are constants. If $\phi(x,\la)$ satisfies
the boundary conditions \eqref{ss*1}, then by the asymptotics
\eqref{sp15} and \eqref{sp12}, we obtain for large positive $\mu$
the asymptotic estimate \bean \phi(x, \la) &=&
C_3\zeta(x)\(\sin\(\mu X\) - \frac{\sin\(\mu\ga\) \sinh\(\mu
X\)}{\sinh\(\mu\ga\)}\)
\[1\]\nonumber\\
          &=& \frac {C_3\zeta(x)}{\sinh\(\mu\ga\)}
          \Big(\sin\(\mu X\) \sinh\(\mu\ga\) -
          \sin\(\mu \ga\) \sinh\(\mu X\)\Big)\[1\],\label{sp9}
\eean and \bean \phi''(x, \la) = -\frac {C_3\mu^2
\zeta(x)}{\sinh\(\mu\ga\)}\(\frac{\rho(x)}{\sigma(x)}\)^{\frac{1}{2}}
          \Big(\sin\(\mu X\) \sinh\(\mu\ga\) +
          \sin\(\mu \ga\) \sinh\(\mu X\)\Big)\[1\],\label{sp50}
\eean where the constant $\ga$ is defined by \be \ga=\int^l_0
\sqrt[4]{\frac{\rho(t)}{\sigma(t)}}
 dt.\label{ga}\ee It is clear
that the eigenvalues $\la_n$ $(n\geq1)$ are the solutions of the
equation $\phi''(l,\la)=0$. Then by \eqref{sp50} , one gets the
following asymptotic characteristic equation
 $$-2\mu^2\zeta(l)\(\frac{\rho(l)}{\sigma(l)}\)^{1/2}\sin\(\mu \ga\) \[1 + \mathcal{O}\(\mu^{-1}\)\]=0,$$
 which can also be rewritten as
\be \sin(\mu \ga)+ \mathcal{O}\(\mu^{-{1}}\)=0.\label{sp5}\ee Since
the solutions of the equation $\sin(\mu \ga) = 0$ are given by
$\widetilde{\mu_{n}}=\frac{n\pi}{\ga},~~n=0,1,2,..., $ it follows
from Rouch\'{e}'s theorem that the solutions of \eqref{sp5} satisfy
the following asymptotic \be \mu_{n} = \widetilde{\mu_{n}} +
\delta_n = \frac{n\pi}{\ga} +\mathcal{O}(n^{-{1}}), ~~\hbox{as}~~n
\rightarrow\infty,\label{mu}\ee which proves \eqref{sp8}.
Furthermore, $\sqrt{\la_n}=
\(\frac{n\pi}{\ga}\)^{2}+\mathcal{O}\(1\),$ and hence
$$\sqrt{\la_{n+1}}-\sqrt{\la_n}\sim\mathcal{O}\(n\), ~~\hbox{as}~~n \rightarrow\infty.$$
The theorem is proved.
\end{Proof}\\
\begin{Prop}\label{Lem22}
The eigenfunctions $(\Phi_{n})_{n\geq1}$ of the spectral problem
\eqref{ss}-\eqref{s1} satisfy the following asymptotic estimates:
\be \Phi_{n}(x)=
\frac{\sqrt{2}\zeta(x)}{\|\zeta\|_{L^2_\rho(0,l)}}\sin
\left\{\frac{n\pi}{\int^l_0 \sqrt[4]
{\frac{\rho(t)}{\sigma(t)}}dt}\int^x_0 \sqrt[4]
{\frac{\rho(t)}{\sigma(t)}}dt\right\}+\mathcal{O}(n^{-1}),\label{sp13}\ee
Furthermore,  \be \displaystyle\lim_{n\rightarrow\infty} \Big{|}
{\la_n^{-\frac{1}{4}}}{\Phi_{n}'(l)}\Big{|}=\frac{\sqrt{2}\zeta(l)}{\|\zeta\|_{L^2_\rho(0,l)}}
\(\frac{\rho(l)}{\sigma(l)}\)^{\frac{1}{4}},\label{sp11}\ee
\begin{equation*}
\end{equation*}where
the function $\zeta$ is defined by \eqref{c10}.
\end{Prop}
\begin{Proof} From  \eqref{sp9} and \eqref{mu},
we obtain the following asymptotic
 estimate for the eigenfunctions
 $\(\Phi_{n}\)_{n\geq1}$:
  \bea \Phi_{n}(x)=
C\zeta(x)\sin \left\{\frac{n\pi}{\ga}\int^x_0 \sqrt[4]
{\frac{\rho(t)}{\sigma(t)}}dt\right\}+\mathcal{O}(n^{-1})\hbox{ as }
n \rightarrow\infty,\eea where $C$ is a constant and $\ga$ is
defined by \eqref{ga}. Taking into account that
$\|\Phi_{n}\|_{L^2_\rho(0,l)} = 1$, a simple computation gives \be
C=\sqrt{2}\|\zeta\|^{-1}_{L^2_\rho(0,l)}+\mathcal{O}(n^{-1}),\label{can}\ee
which proves \eqref{sp13}.\\
In a similar way, from the asymptotics \eqref{sp12}, \eqref{mu} and
\eqref{can}, a straightforward computation yields \bea \Phi_{n}'(x)=
\sqrt{2}\|\zeta\|^{-1}_{L^2_\rho(0,l)}\zeta(x)
\(\frac{n\pi}{\ga}\)\(\frac{\rho(x)}{\sigma(x)}\)^{\frac{1}{4}}\cos
\left\{\frac{n\pi}{\ga}\int^x_0 \sqrt[4]
{\frac{\rho(t)}{\sigma(t)}}dt\right\}+\mathcal{O}(1), \eea and hence
\bea \mid\Phi_{n}'(l)\mid\sim
\sqrt{2}\|\zeta\|^{-1}_{L^2_\rho(0,l)}\zeta(l)
\(\frac{n\pi}{\ga}\)\(\frac{\rho(l)}{\sigma(l)}\)^{\frac{1}{4}},
\hbox{ as } n \rightarrow\infty. \eea
The proof is complete.
\end{Proof}
\section{Exact controllability for the Linear Problem}
The goal of this section is to prove the exact controllability for
the linear control problem \eqref{crepeau9}. To this end, we first
prove the observability results which are consequence of the
spectral properties given in Section \ref{Spe}. 
\begin{Prop}\label{ph1}
Let $T
>0$ and  $(y^0,{y}^1)\in H^1_0(0,l) \times
H^{-1}(0,l)$. Then
\begin{equation}\label{ob}
\int_{0}^{T}|{y}_x(t,l)|^{2}dt \asymp
\|({y}^0,{y}^1)\|_{H^1_0(0,l)\times H^{-1}(0,l)}^{2},
\end{equation}
where $y$ is the solution of Problem \eqref{crepeau10}.
\end{Prop}
In order to prove Proposition \ref{ph1}, we need the following
variant of Ingham's inequality due to Haraux \cite{AH}.
\begin{Lem}\cite{AH}\label{In}
Let $f(t)=\sum\limits_{n\in\mathbb{Z}}c_ne^{-i\tilde{\la}_nt}$,
where $(\tilde{\la}_n)_{n\in\mathbb{Z}}$ is a sequence of real
numbers. We assume that there exist $N\in\mathbb{N}$, $\beta> 0$ and
$\varrho> 0$ such that \be\label{in2}
|\tilde{\la}_{n+1}-\tilde{\la}_n|>\varrho,~~\hbox{if }|n|>N,\ee and
$|\tilde{\la}_{n+1}-\tilde{\la}_n|>\beta$, for all $n\in\mathbb{Z}$.
Then for any $T>\dfrac{\pi}{\varrho}$,
\begin{equation*}
\int_{0}^{T}|f(t)|^{2}dt \asymp\sum\limits_{n\in
\mathbb{Z}}|c_{n}|^{2},
\end{equation*}
for all sequences of complex numbers $(c_n)_{n\in\mathbb{Z}}\in
{\ell}^2$.
\end{Lem}
\begin{proof} {\bf of Proposition \ref{ph1}.} Let us first recall
from the spectral representation of the space $\mathcal{H}_{\theta}$
that
$$\mathcal{H}_{1/4}=
\{u(x)=\sum\limits_{n\in\N\backslash\{0\}}c_n\Phi_n(x)~:
 ~\|u\|_{\theta}^2=\sum\limits_{n\in\N\backslash\{0\}}\la_n^{1/2}|c_n|^{2}<\infty\}$$
 and
 $$\mathcal{H}_{-1/4}=
\{u(x)=\sum\limits_{n\in\N\backslash\{0\}}c_n\Phi_n(x)~:
 ~\|u\|_{\theta}^2=\sum\limits_{n\in\N\backslash\{0\}}\la_n^{-1/2}|c_n|^{2}<\infty\},$$
where the eigenfunctions $\(\Phi_{n}\)_{n\geq1}$ are defined in
Corollary \ref{cor}. Since the coefficients $\rho$, $\si$ and $q$
are bounded in $\[0,l\]$, then from the asymptotic estimates
\eqref{sp8} and \eqref{sp13}, one gets
$$\mathcal{H}_{1/4}=H_0^1\(0,l\)~\hbox{ and }~\mathcal{H}_{-1/4}=H^{-1}\(0,l\).$$
On the other hand, by \eqref{sl}, one has  \bea
\int_{0}^{T}|{y}_x(t,l)|^{2}dt =
\int_{0}^{T}\Big{|}\sum\limits_{n\in \mathbb{Z}\backslash\{0\}}
c_{n}e^{-i\tilde{\la}_n t} \Phi_n'\(l\)\Big{|}^{2},\eea where
$\tilde{\la}_{n}=\sqrt{\lambda_{n}},$ and
\bean\label{ob1000}\left\{
\begin{array}{ll}
\tilde{\la}_{-n}=-\tilde{\la}_{n}\hbox{~and } \Phi_n=\Phi_{-n},~
n\in
\mathbb{N}\backslash\{0\},\\
c_n=\overline{c_{-n}}=\frac{1}{2}\Big(a_n-i\frac{b_n}{\sqrt{\tilde\la_n}}\Big),~n\in
\mathbb{N}\backslash\{0\}.
\end{array}
 \right.
\eean Therefore, in view of the first statement of Theorem
\ref{Lem2} and the gap condition \eqref{g1}, Lemma \ref{In} implies
that for every $T>0$ \be \int_{0}^{T}|{y}_x(t,l)|^{2}dt
\asymp\sum\limits_{n\in \mathbb{Z}\backslash\{0\}}
|c_{n}\Phi'_n(l)|^2= \frac{1}{2}\(\sum\limits_{n\in
\mathbb{N}\backslash\{0\}}|a_n \Phi'_n(l)|^{2}+\Big{|}
\frac{b_n}{\sqrt{\la_n}}\Phi'_n(l)\Big{|}^{2}\).\label{sl1}\ee 
Furthermore, by the second statement of Theorem \ref{Lem2}, we have
$\Phi'_n(l)\not\equiv 0$ for all $n\in \mathbb{N}\backslash\{0\}$.
Thus by \eqref{sp11}, there exists $m,~ M > 0$ such that $$m
\sqrt[4]{\la_n}<|\Phi'_n(l)|<M \sqrt[4]{\la_n}.$$ Consequently,
$$\int_{0}^{T}|{y}_x(t,l)|^{2}dt \asymp
\sum\limits_{n\in \mathbb{N}\backslash\{0\}}\sqrt{\la_n}|a_n|^2+
\sum\limits_{n\in
\mathbb{N}\backslash\{0\}}\frac{|b_n|^2}{\sqrt{\la_n}}=
\|({y}^0,{y}^1)\|_{H^1_0(0,l)\times H^{-1}(0,l)}^{2},$$ which proves
\eqref{ob}. 
 This completes the proof.

\end{proof}\\
Let us now state the existence and uniqueness result for the control
system \eqref{crepeau9}. Following \cite[Theorem 2.14]{V.K} and
\cite{C}, we define a weak solution to the control system
\eqref{crepeau9} using the method of transposition.
\begin{Prop} \label{b-p} Let
$T>0$, and  $u\in L^{2}(0,T)$. For any $(y^0, y^1)\in H^1_{0}(0,l)
\times H^{-1}(0,l)$, there exists a unique weak solution $y$ of
System \eqref{crepeau9} in the class
$$(y,y_t)\in C\([0, T ]; H^1_{0}(0,l) \times H^{-1}(0,l)\).$$
Moreover, there exists a constant $C>0$ such that
$$\|(y,y_t)\|_{H^1_{0}(0,l) \times H^{-1}(0,l)}\leq
C\(\|(y^0,y^1)\|_{H^1_{0}(0,l) \times
H^{-1}(0,l)}+\|u\|_{L^{2}(0,T)}\).$$\end{Prop} This result is
basically well known for $\rho=\si=q=1$ (see \cite[Proposition
10]{C}). The proof can be also easily extended to the variable
coefficient case.\\
We are now ready to state the main result of this section. Notice
that, in view of the fact that \eqref{crepeau9} is linear and
reversible in time, this system is exactly controllable if and only
if the system is null controllable.
\begin{Theo}
\label{hj} Assume that the coefficients $\rho$, $\sigma$ and $q$
satisfy \eqref{crepeau7} and \eqref{crepeau8}. Given $T>0$ and
 $(y^0,y^1)\in H^{-1}(0,l)\times H^1_0(0,l)$, there
exists a control $u\in L^{2}(0,T)$ such that the solution $y$ of the
control problem \eqref{crepeau9}
 satisfies
\begin{equation*}
y(T,x)=y_t(T,x)=0,~~x\in\[0,l\].
\end{equation*}
\end{Theo}
\begin{proof} 
 Following \cite[Proposition 11]{C}, we apply the
Lions'{\rm HUM} \cite{J.L2}, then the control problem is reduced to
the obtention of the observability inequalities \eqref{ob} for the
uncontrolled system \eqref{crepeau10}. Therefore, Theorem \ref{hj}
immediately follows from Proposition \ref{ph1}.

\end{proof}
\section{Controllability for the Nonlinear Problem}
In this section we prove the local exact controllability for the
nonlinear control problem \eqref{crepeau6}. First of all, we
introduce the following space: \be\mathcal{H}=\{y\in H^3(0,l)\hbox{
such that } y(0) =y_{xx}(0) = y(l)=0\}.\ee The main result of this
section is stated as follows:
\begin{Theo}\label{nont}
Let $T>0$ and assume that the coefficients $\rho$, $\sigma$ and $q$
satisfy \eqref{crepeau7} and \eqref{crepeau8}. Then there exists a
constant $r > 0$ such that for all initial and final conditions $
\(y^{0}, y^{1}\)$, $\(y_T^0,y_T^1\) \in\mathcal{H}\times
H^1_0(0,l)$, with $$\|\(y^{0}, y^{1}\)\|_{\mathcal{H}\times
H^1_0(0,l)} < r\hbox{ and }\|\(y_T^0,y_T^1\)\|_{\mathcal{H}\times
H^1_0(0,l)} < r,$$ there exists a control $u \in H^1(0, T)$ such
that the solution $y\in C\([0, T];H^3(0,l)\)\cap C^1\([0,
T];H^1_0(0,l)\)$ of Problem \eqref{crepeau6} satisfies  \be y(T, x)=
y_T^0,~~y_{t}(T,x)= y_T^1, ~~ x\in\[0,l\]. \ee
\end{Theo}
In order to prove Theorem \ref{nont}, we need the following result
whose the proof is similar to that \cite[Proposition 12]{C}.
\begin{Prop}\label{propH} Assume that the coefficients $\rho$, $\sigma$ and $q$ satisfy
\eqref{crepeau7} and \eqref{crepeau8}. Given $T>0$ and $(y^0,y^1)\in
\mathcal{H}\times H^1_0(0,l)$, then
 there exists a control $u\in H^{1}(0,T)$ such that the solution
 $y\in C\([0, T ], \mathcal{H}\)\cap C^1\([0, T];H^1_0(0,l)\)$ of
 the linear
control problem \eqref{crepeau9}, satisfies
\begin{equation}\label{xx}
y(T,x)=y_t(T,x)=0,~~x\in[0, l].
\end{equation}
\end{Prop}
\begin{Rem} \label{rem}  In view of
the fact that the linearized problem \eqref{crepeau9} is reversible
in time and exactly controllable in the space $\mathcal{H}\times
H^1_0(0,l)$. Then there exists a continuous linear map
$$\Gamma:~~\(\hat y^0_T,\hat y^1_T\)\in \mathcal{H}\times
H^1_0(0,l)\rightarrow u \in H^1(0,T)$$ such that the solution $\hat
y$ of the problem  \bean \label{n2}
 \left\{
 \begin{array}{ll}
\rho(x)\hat y_{tt}+(\sigma(x)\hat y_{xx})_{xx}-(q(x)\hat y_x)_x=0,&(t,x)\in(0,T)\times(0,l),\\
\hat y(t, 0) = \hat y_{xx}(t,0) = \hat y(t, l) = 0,~~ \si(l)\hat y_{xx}(t, l) = u(t),&t\in(0,T),\\
\hat y(0, x) = 0,~~\hat y_{t}(0, x) = 0,&x\in(0,l),\\
\end{array}
 \right.
\eean satisfies $\(\hat y(T),\hat y_t(T)\)=\( \hat y^0_T, \hat
y^1_T\)$. Furthermore there exists a constant $k>0$ such that, \bean
\label{n6}\|\Gamma\|_{H^1(0,T)}\leq k\|\( \hat y^0_T, \hat
y^1_T\)\|_{\mathcal{H}\times H^1_0(0,l)}\eean
\end{Rem} Now, let us denotes by $\psi_0$ and $\psi_1$, the following maps which are linear and continuous by
Proposition \ref{propH}:
                                         $$\psi_0:~
                                         \(\bar y^0,\bar y^1\)\in \mathcal{H}\times
H^1_0(0,l)\rightarrow \bar y \in C\([0, T ], \mathcal{H}\)\subset
L^2\(0,T; \mathcal{H}\),$$ where $\bar y$ is the solution of the
problem \bean \label{n1}
 \left\{
 \begin{array}{ll}
\rho(x)\bar y_{tt}+(\sigma(x)\bar y_{xx})_{xx}-(q(x)\bar y_x)_x=0,&(t,x)\in(0,T)\times(0,l),\\
\bar y(t, 0) = \bar y_{xx}(t,0) = \bar y(t, l) = 0,~~ \bar y_{xx}(t, l) = 0,&t\in(0,T),\\
\bar y(0, x) = \bar y^{0},~~\bar y_{t}(0, x) = \bar y^{1},&x\in(0,l),\\
\end{array}
 \right.
\eean and $$\psi_1:~
                                         u\in H^1(0,T)\rightarrow \hat y
                                          \in C\([0, T ], \mathcal{H}\)\subset
L^2\(0,T; \mathcal{H}\),$$ where $\hat y$ is the solution of the
problem \eqref{n2}. Moreover, there exist two constants $k_0,k_1>0$
such that \be \|\psi_0\|_{L^2\(0,T; \mathcal{H}\)}\leq
k_0\|\(y^0,y^1\)\|_{\mathcal{H}\times H^1_0(0,l)} ~\hbox { and }~
\|\psi_1\|_{L^2\(0,T; \mathcal{H}\)}\leq
k_1\|u\|_{H^1(0,T)}.\label{n50}\ee Let us recall also that if $f\in
L^1\(0,T; H^1(0,l)\)$, then the following problem \bean\label{n3}
 \left\{
 \begin{array}{ll}
\rho(x)\tilde y_{tt}+(\sigma(x)\tilde y_{xx})_{xx}-(q(x)\tilde y_x)_x=f,&(t,x)\in(0,T)\times(0,l),\\
\tilde y(t, 0) = \tilde y_{xx}(t,0) = \tilde y(t, l) = 0,~~ \tilde y_{xx}(t, l) = 0,&t\in(0,T),\\
\tilde y(0, x) = 0,~~\tilde y_{t}(0, x) = 0,&x\in(0,l),\\
\end{array}
 \right.
\eean
 has a unique weak solution $\(\tilde y,\tilde y_t\) \in C\([0, T ],\mathcal{H}\times H^1(0,l)\)$.
 As consequence, the linear map

                                         $$\psi_2:~
                                         f\in L^1\(0,T; H^1(0,l)\)
                                         \rightarrow \(\tilde y,\tilde y_t\) \in C\([0, T ],
                                          \mathcal{H}\times H^1(0,l)\)$$
                                          is continuous and there exists a constant $k_2>0$, such
that \bean \label{n4}\|\psi_2\|_{C\([0, T ], \mathcal{H}\times
H^1(0,l)\)}\leq k_2\|f\|_{L^1\(0,T; H^1(0,T)\)}.\eean Before we
prove Theorem \ref{nont}, we need the following result due to
\cite{C}.
\begin{Prop}\label{propc}\cite[Proposition 13]{C}
The map $$y\in L^2([0, T],\mathcal{H})\rightarrow(y^2)_{xx}\in
L^1([0, T],H^1(0,l))$$ is well-defined and continuous. Furthermore,
there exist a constant $k_3>0$ such that, \bean\label{n5}
\|(y^2)_{xx}-(z^2)_{xx}\|_{L^1([0, T],H^1(0,l))}\leq
k_3\(\|y+z\|_{L^2([0, T],\mathcal{H})}\|y-z\|_{L^2([0,
T],\mathcal{H})}\)\eean
\end{Prop}
We are now ready to prove Theorem \ref{nont}.\\
\begin{Proof} Consider the nonlinear problem
 \eqref{crepeau6} with initial data $\(y^{0}, y^{1}\)\in{\mathcal{H}\times
H^1_0(0,l)}$ and a control $u\in H^1(0, T)$. Let $\bar y$ and
$\tilde y$ be the solutions of \eqref{n1} and \eqref{n3},
respectively, where $\(\bar y^{0},
 \bar y^{1}\)\equiv \(y^0,y^1\)$ and $f\equiv(y^2)_{xx}$. Obviously, the solution
 $y$ of the nonlinear problem
 \eqref{crepeau6}, can be written in the form $y:=\hat y+\bar y+ \tilde
 y$, where $\hat y$ is the solution \eqref{n2}.
Let $\mathfrak{F}$ be the nonlinear map : \bea\mathfrak{F}: y\in
L^2\(0,T;\mathcal{H}\)\rightarrow \mathfrak{F}(y) \in
L^2\(0,T;\mathcal{H}\) \eea such that \bea
&~&\mathfrak{F}(y)=\psi_0(y^0,y^1)-\psi_2\((y^2)_{xx}\)\\
&~&+\psi_1\circ\Gamma\Big{(}(y^0_T,y^1_T)-
\(\psi_0(y^0,y^1)(T),\psi_{0t}(y^0,y^1)(T)\)-\(\psi_2((y^2)_{xx})(T),\psi_{2t}((y^2)_{xx})(T)\)\Big{)}.
\eea To prove the theorem it suffices to show that $\mathfrak{F}$
has a fixed point. Furthermore, using \eqref{n6}, \eqref{n50},
\eqref{n4} and \eqref{n5}, it can be shown by a straightforward
computation that there exists a constant $\tilde k>0$ such that
\bean \|\mathfrak{F}(y)\|_{L^2\(0,T;\mathcal{H}\)}\leq \tilde
k\(\|y\|^2_{L^2([0, T],\mathcal{H})}+ \|\(y^{0},
y^{1}\)\|_{\mathcal{H}\times
H^1_0(0,l)}+\|\(y_T^0,y_T^1\)\|_{\mathcal{H}\times
H^1_0(0,l)}\)\label{n7}\eean  and
\be\|\mathfrak{F}(y)-\mathfrak{F}(z)\|_{L^2\(0,T;\mathcal{H}\)} \leq
\tilde k\|y-z\|_{L^2([0, T],\mathcal{H})}\(\|y\|_{L^2([0,
T],\mathcal{H})}+\|z\|_{L^2([0, T],\mathcal{H})}\).\label{n60}\ee
Let $ \(y^{0}, y^{1}\)$, $\(y_T^0,y_T^1\) \in\mathcal{H}\times
H^1_0(0,l)$ such that $$\|\(y^{0}, y^{1}\)\|_{\mathcal{H}\times
H^1_0(0,l)} < r \hbox{ and }\|\(y_T^0,y_T^1\)\|_{\mathcal{H}\times
H^1_0(0,l)} <r ,$$ where $r
> 0$ is a constant which will be fixed later. For each $R >
0$, let us denote the ball of radius $R$ and centered at the origin
by
 $$B\(0,R\) := \{y\in
L^2\(0, T ; \mathcal{H}\), \hbox{ such that } \|y\|_{L^2(0,T
,\mathcal{H})}\leq R\}.$$ Since we aim to use Banach fixed point
theorem to the restriction of $\mathfrak{F}$ to the ball
$\overline{B}\(0,R\)$, then the constants $r
> 0$ and $R> 0$ can be chosen in \eqref{n7} and \eqref{n60}, such that $\hat k(2r + R^2) \leq R$ and
$2R\hat k<1$. Let $R=\frac{1}{4\hat k}$ and $r=\frac{3R^2}{2}$, then
by \eqref{n7} and \eqref{n60}, one gets \bean
\mathfrak{F}\(B(0,R)\)&\subset&
B(0,R),\label{61}\\
\|\mathfrak{F}(y)-\mathfrak{F}(z)\|_{L^2\(0,T;\mathcal{H}\)} &\leq&
\frac{1}{2}\|y-z\|_{L^2([0, T],\mathcal{H})},\label{62}\eean which
implies by Banach fixed point theorem that $\mathfrak{F}$ has a
unique fixed point. The proof of the theorem is completed.
\end{Proof}
 
\vskip 3cm
\end{document}